\documentclass[12pt]{article}
\usepackage{}

\usepackage{epsfig}
\usepackage{latexsym}
\usepackage{caption}
\usepackage{amsfonts}
\usepackage{amssymb}
\usepackage{mathrsfs}
\usepackage{amsmath}
\usepackage{enumerate}
\usepackage{graphics}
\usepackage{MnSymbol}
\usepackage{float}
\usepackage{pict2e}
\usepackage{subfigure}

\makeatletter

\newcommand{\Rmnum}[1]{\expandafter\@slowromancap\romannumeral #1@}

\makeatother

	\newtheorem{theorem}{Theorem}
	\newtheorem{observation}[theorem]{Observation}

	\newtheorem{guess}[theorem]{Conjecture}

	\newtheorem{question}[theorem]{Question}
	\newtheorem{lemma}[theorem]{Lemma}

	\makeatother

	\newenvironment{proof}{\noindent {\bf
			Proof.}}{\rule{3mm}{3mm}\par\medskip}

	\newcommand{\q}{\uppercase\expandafter{\romannumeral1}}
	\newcommand{\qq}{\uppercase\expandafter{\romannumeral2}}
	\newcommand{\qqq}{\uppercase\expandafter{\romannumeral3}}
	\newcommand{\qqqq}{\uppercase\expandafter{\romannumeral4}}
	\newcommand{\qqqqq}{\uppercase\expandafter{\romannumeral5}}
	\newcommand{\qqqqqq}{\uppercase\expandafter{\romannumeral6}}

\begin{document}
		
\title{Relatively small counterexamples to Hedetniemi's conjecture  }
\author{ 
	Xuding Zhu\thanks{Zhejiang Normal University, 	
		email:xdzhu@zjnu.edu.cn. 
		Grant Numbers: NSFC 11971438 and 111 project of Ministry of Education of China.}}


	\maketitle
	\begin{abstract}
		
		Hedetniemi conjectured in 1966 that $\chi(G \times H) = \min\{\chi(G), \chi(H)\}$
		for all graphs $G$ and $H$.  Here $G\times H$ is the graph with vertex set
	$	V(G)\times V(H)$ defined by putting $(x,y)$ and $(x',y')$ adjacent if and only if
		$xx'\in E(G)$ and $yy'\in E(H)$. This conjecture received a lot of attention in the past half century. 
	Recently, Shitov refuted this conjecture. 
		   Let $p$ be the minimum number of vertices in a graph of odd girth $7$ and fractional chromatic number greater than $3+4/(p-1)$.
		 Shitov's proof  shows that Hedetniemi's conjecture fails for some graphs with chromatic number about  $p^33^p $ and with  about $(p^33^p)^{p^43^{p-1}} $ vertices. In this paper,  we  show  that the conjecture fails already for some graphs $G$ and $H$  with chromatic number $3\lceil \frac {p+1}2 \rceil $ and with $p \lceil (p-1)/2 \rceil$ and   $3  \lceil \frac {p+1}2 \rceil   (p+1)-p$ vertices, respectively. The currently known upper bound for $p$ is  $83$. Thus Hedetniemi's conjecture fails for some graphs $G$ and $H$ with chromatic number $126$, and with  $3,403$ and $10,501$ vertices, respectively.   
	\end{abstract}
	
	\section{Introduction}

	The \emph{ product}    $G \times H$ of  graphs $G$ and $H$  has vertex set $V(G) \times V(H)$ and has  $(x, y)$ adjacent to $(x',y')$   if and only if $xx' \in E(G)$ and $yy' \in E(H)$.  Many names for this product are used in the literature, including   the \emph{categorical product}, the \emph{tensor product} and the 
	\emph{direct product}. It is the most important product in this   note. We just call it \emph{the product}. We may write $x \sim y$ (in $G$)
	to denote $xy \in E(G)$.
	
	 A proper colouring $\phi$ of $G$  induces a proper colouring $\Phi$  of $G \times H$ defined as $\Phi(x,y) = \phi(x)$. So $\chi(G \times H) \le   \chi(G)$.  Symmetrically, we also have 
	 $\chi(G \times H) \le \chi(H)$. Therefore $\chi(G \times H) \le \min \{\chi(G), \chi(H)\}$.
	 In 1966, 
	 Hedetniemi conjectured in  \cite{Hedetniemi} that equality always holds in the above inequality. 
	 
	 \begin{guess}[Hedetniemi's conjecture]
	 	\label{hj}
	 	For any positive integer $c$, if $G \times H$ is $c$-colourable, then at least one of $G$ and $H$ is $c$-colourable.
	 \end{guess}

	 This conjecture received a lot of attention in the past half century  (see \cite{ES,Klav, Sauer,Tardif,Zhu,Zhu-frac}). Some special cases are confirmed.  In particular, it was proved by El-Zahar and Sauer \cite{ES} that Hedetniemi's conjecture holds for $c = 3$ (where for $ c \le 2$, the conjecture holds trivially).   
	 Also, it was proved in \cite{Zhu-frac} that a fractional version of Hedetniemi's conjecture is true, i.e., for any graphs $G$ and $H$, 
	 $\chi_f(G \times H) = \min \{\chi_f(G), \chi_f(H)\}$.
	 
	 Recently,  Shitov   refuted Hedetniemi's conjecture    \cite{Shitov}. He proved that 
	 Hedetniemi's conjecture fails for sufficiently large $c$. Let $p$ be the minimum number of vertices in a graph $G$ of odd girth $7$ and fractional chromatic number greater than $3+4/(p-1)$. Shitov's proof shows that Hedetniemi's conjecture fails for some  $c$ about $ 3^pp^3$ and for graphs with about $c^{p^32^{p-1}}$ vertices. 
	 The current known upper bound for $p$  is $83$ \cite{Tardif4}. 
	Thus Shitov's result shows that Hedetniemi's conjecture fails for some $c$  that is about $ 3^{95}$ and for graphs with about $(3^{95})^{3^{99}}$ vertices.
	  
	  On the other hand, we do not know if Hedetniemi's conjecture holds for any integer $c \ge 4$. A natural question is whether   Hedetniemi's conjecture fails for relatively small  $c$. 
	  
	  This paper shows that Hedetniemi's conjecture fails for some graphs $G$ and $H$ with chromatic number  $3\lceil (p+1)/2 \rceil $ and with   $p \lceil (p-1)/2 \rceil$ and   
	  
	 Using the upper bound $p \le 83$, we conclude that  Hedetniemi's conjecture fails for  $c = 125$ (and hence $G$ and $H$ can be assumed to have chromatic number $126$). The number of vertices in $G$ and $H$ are   $3,403$ and $10,501$, respectively.

	  	\section{Exponential graph}
	  	
	  	One of the standard tools used in the study of Hedetniemi's conjecture is the concept of \emph{exponential graphs}.  Let $c$ be a positive integer. We denote by $[c]$ the set 
	  	$\{1,2,\ldots, c\}$, and for  integers $c \le  d$, let  $[c,d]=\{c, c+1, \ldots, d\}$.    For a graph $G$, the exponential graph 
	  	$K_c^G$ has  vertex set $$\{f: f \text{ is a mapping from $V(G)$ to $[c]$}\},$$
	  	with $fg \in E(K_c^G)$ if and only if for any edge $xy \in E(G)$, $f(x) \ne g(y)$. 
	  	In particular, $f\sim f$ is a loop in $K_c^G$ if and only if $f$ is a proper $c$-colouring of $G$.
	  	So if $\chi(G) > c$, then $K_c^G$ has no loop. 
	  	

	  	For two graphs $G$ and $H$, a \emph{homomorphism from $G$ to $H$} is a mapping $\phi: V(G) \to V(H)$ that preserves edges, i.e., for every edge $xy$ of $G$, $\phi(x)\phi(y)$ is an edge of $H$. We say $G$ is \emph{homomorphic} to $H$, and write $G \to H$,  if there is a homomorphism from $G$ to $H$. The ``homomorphic" relation ``$\to$"  is a quasi-order. It is reflexive and transitive: if $G \to H$ and $H \to Q$ then $G \to Q$.  The composition $\psi \circ \phi$ of a homomorphism 
	  	$\phi$ from $G$ to $H$ and a homomorphism $\psi$ from $H$ to $Q$ 
	  	is a homomorphism from $G$ to $Q$. 
	  	
	  	Note that a homomorphism from a graph $G$ to $K_c$ is equivalent to a proper $c$-colouring of $G$. Thus if $G \to H$, then $\chi(G) \le \chi(H)$. The following result was proved in \cite{ES}. For the completeness of this paper, we include a short proof.
	  	
	  	\begin{lemma}[\cite{ES}]
	  		\label{lem-homo}
	  		For any graph $H$, 
	  		$\chi(G \times H) \le c$ if and only if   $H$ is homomorphic to $K_c^G$. 
	  	\end{lemma}
	  	\begin{proof}
	  		The maps $\Psi:V(G)\times V(H)\to[c]$ naturally correspond to the maps $f: V(H)\to V(K_c^G)$ via $f(u)(v)=\Psi(v,u)$. This is a 1-1 correspondence where proper colorings correspond to exactly the homomorphisms, i.e., $\Psi$ is a proper colouring if and only if $f$ is a homomorphism from $F$ to $K_c^G$. 
	  	\end{proof}

	  	So $K_c^G$ is the largest graph $H$ in the order of homomorphism with the property 
	  	that $\chi(G \times H) \le c$.  
	  	Thus  Conjecture \ref{hj}  is equivalent to the following statement:
	  	
	  	\bigskip
	  	(*)	{\em For any positive integer $c$, if $\chi(G) >c$, then $\chi(K_c^G) \le c$.}
	  	\bigskip

	  	\section{Construction of a counterexample}

	  	The {\em lexicographic product} $G[H]$ of two graphs $G$ and $H$ is the graph obtained from $G$ by replacing each vertex of $G$  by a copy of $H$. Thus the graph $G[K_q]$ has vertex set  $\{ (x, i): x \in V(G), i \in [q]\}$, where $(x,i)$ and $ (y,j)$ are adjacent  if and only if 
	  	either $xy \in E(G)$ or $x=y$ and $i \ne j$. 
	  	The 
	  	{\em fractional chromatic number}  $\chi_f(G)$ of $G$ is defined as 
	  	$$\chi_f(G) = \inf \{ \chi(G[K_q])/q: q = 1,2,\ldots \}.$$
	  	The {\rm odd girth} of  $G$ is the length of a shortest odd cycle in $G$. 
	  	
	  	Let $p$ be the minimum order of a graph of odd girth $7$ and with fractional chromatic number greater than $3+ \frac{4}{p-1}$. 
	  	Let $F$ be such a graph. 
	  	The existence of $F$ is guaranteed by  a classical result of Erd\H{o}s \cite{Erdos} that there are graphs $F$ with $girth(F) \ge g$ and $\chi_f(F) \ge r$ for any   $g,r$ (usually it is stated as $\chi(F) \ge r$, but the   proof actually shows that $\chi_f(F) \ge r$). In Section \ref{sec4}, we shall discuss more on the order $p$ of $F$. We assume  $V(F) = \{v_1, v_2, \ldots, v_p\}$.

	  Let $q \ge (p-1)/2$ be an integer and $c =3q+2$. Let $G=F[K_q]$. Let $H$ be the graph that consists of 
	  \begin{enumerate}
	  	\item[(1)]
	   vertices $g_i$ for $i\in[c]$ forming a clique;
	  	\item[(2)] a vertex $\phi$ adjacent to $g_i$ for $i>p$;
	  	\item[(3)] for each $i\in[p]$, for each $t\in[q+2,3q+2]$, 
	  	 a vertex $\mu_{i,t}$, where for each 
	  	 $i\in[p]$, the set $\{\mu_{i,t}: t \in [q+2,3q+2]\} \cup \{g_i\}$   forms a clique,  and the vertex $\mu_{i,t}$ is further adjacent to  $g_j$ with $j>2q+1, j\ne t$;  
	  	\item[(4)] for each $i\in [p]$ and $t\in[2q+2,3q+2]$,  a vertex $\theta_{i,t}$ adjacent to $\phi$ and $\mu_{i,t}$ and   to each $g_j$ with $j\notin\{i,t\}$.
	  \end{enumerate}

	  \begin{figure}[!htb]
	  	\centering
	  	\includegraphics[scale=0.8]{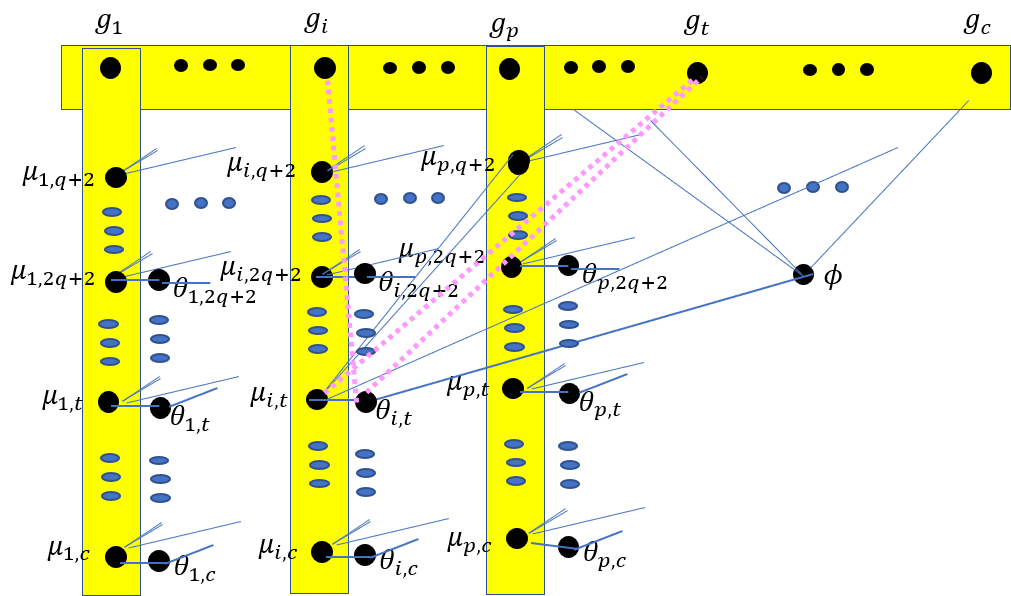}
	  	\caption{The graph $H$, where each shaded rectangle is a clique}\label{fig1}
	  \end{figure}

	  \begin{theorem}
	  	\label{thm-main}
	  	$\chi(G), \chi(H) > c$ and $\chi(G \times H) \le c$.
	  \end{theorem}
	  \begin{proof}
	It follows from the definition of fractional chromatic number that 
	  $$\chi(G) = \chi(F[K_q]) \ge q \chi_f(F) > 3q+2 =c.$$
	  Next we show that $\chi(H) > c$. Assume to the contrary that $\chi(H) \le c$ and $\Psi$ is a $c$-colouring of $H$. We may assume that $\Psi(g_i)=i$. 
	  
	  As  $\phi$ is adjacent to $g_j$ for $ j > p$, we conclude that   $\Psi(\phi) = i^*$ for some $i^* \le p$. 
	  
	  As the set   $\{ \mu_{i^*,t}:   t \in [q+2,3q+2]\} \cup \{g_{i^*}\}$    forms a clique of $2q+2$ vertices and $\Psi(g_{i^*})=i^* \le 2q+1$,  there exists $t^* \in [2q+2, 3q+2]$ such that
	  $\Psi(\mu_{i^*,t^*}) \ge 2q+2$. As $\mu_{i^*,t^*}$ is adjacent to 
	  all $g_j$ with $j \in [2q+2, 3q+2]- \{t^*\}$, we conclude that $\Psi(u_{i^*,t^*}) = t^*\ge 2q+2$.
	  
	    The vertex  $\theta_{i^*,t^*}$ is adjacent to each $g_j$ with $j \notin \{i^*,t^*\}$. Hence 
	    $\Psi( \theta_{i^*,t^*}) \in   \{i^*,t^*\}$. But $\theta_{i^*,t^*}$ is adjacent to $\phi$, which is coloured by $i^*$, and adjacent to $\mu_{i^*,t^*}$, which is coloured by $t^*$. This is a contradiction, and thus $\chi(H) > c$. 
	    
	  It remains to show that $\chi(G \times H) \le c$. By Lemma \ref{lem-homo}, it suffices to show that $H$ is homomorphic to $K_c^G$. We shall prove that $H$ is actually a subgraph of $K_c^G$. 
	  For this purpose, we simply define each vertex $y$ of $H$ as a mapping   $y: V(G) \to [c]$ and show that if $yy' \in E(H)$, then $yy' \in E(K_c^G)$, i.e., for any edge $xx' \in E(G)$, 
	  $y(x) \ne y'(x')$. (This is equivalent to say that the mapping $\Phi: V(G \times H) \to [c]$ defined as  $\Phi(x,y) = y(x)$ is a proper $c$-colouring of $G \times H$). 
	  
	    Each vertex of $G$ is of the form $(v_s,j)$, $s \in [p]$ and $j \in [q]$.
	   
	   For  $y \in V(H)$, let  $y: V(G) \to [c]$ be defined as follows:
	   \begin{enumerate}
	   	\item For $i \in [c]$ and $(v_s,j) \in V(G)$,  $g_i(v_s,j)= i$.
	   	\item  For $(v_s,j) \in V(G)$,    $\phi(v_s,j) = s$.
	   	\item  For  $i\in [p]$, $t\in[q+2,3q+2]$ and 
	   	$(v_s,j) \in V(G)$,  
	   	  \[
	   	\mu_{i,t}(v_s,j) = \begin{cases} 
	   	j+ \delta_{j \ge i}, &\text{ if $d_F(v_s, v_i) = 0$ or $2$ }, \cr
	   	q+j +  \delta_{q+j \ge i}, &\text{ if   $d_F(v_s, v_i) = 1$ }, \cr
	   	t-\delta_{i \ge t},  &\text{ if   $d_F(v_s, v_i) \ge 3$,  } \cr
	   	\end{cases}
	   	\]
	   	where 
	   	\[
	   	\delta_{j \ge i} = \begin{cases} 1, &\text{ if $j \ge i$}, \cr
	   	0, &\text{otherwise.} \cr
	   	\end{cases}
	   	\]
	   	\item For   $i\in[p]$, $t\in[2q+2,3q+2]$ and 
	   	 $(v_s,j) \in V(G)$,  \[
	   	\theta_{i,t}(v_s,j)= \begin{cases} 
	   	i, &\text{ if $d_F(v_s, v_i) \ge 2$ }, \cr
	   	t, &\text{ if   $d_F(v_s, v_i) \le  1$. }  \cr
	   	\end{cases}
	   	\]
	   \end{enumerate}

	   Now we show that if $yy' \in E(H)$, then $yy' \in E(K_c^G)$. For $ y \in V(H)$, let 
	   $$Im(y) = \{y(x): x \in V(G)\}.$$
	   Recall that $yy' \in E(K_c^G)$ if and only if for every edge $xx' \in E(G)$, $y(x) \ne y'(x')$. So we have the following easy observation.
	   
	   \begin{observation}
	   	\label{ob-easy}
	   	If $Im(y) \cap Im(y') = \emptyset$, then $yy' \in E(K_c^G)$. 
	   \end{observation}

	   \begin{observation}
	   	\label{ob3}
	   	   For  $i\in [p]$, $t\in[q+2,3q+2]$,
	   	   the following hold:
	   	   \begin{itemize}
	   	   	\item 	If 
	   	   	$d_F(v_s, v_i)=0$ or $2$, then $\mu_{i,t}(v_s,j) < i$ when $j < i$ and $\mu_{i,t}(v_s,j) > i$ when $j \ge i$.
	   	   	\item  If 
	   	   	$d_F(v_s, v_i)=1$, then $\mu_{i,t}(v_s,j) < i$ when $q+j < i$ and $\mu_{i,t}(v_s,j) > i$ when $q+j \ge i$.
	   	   	\item If 
	   	   	$d_F(v_s, v_i)\ge 3$, then $\mu_{i,t}(v_s,j) < i$ when $t \le  i$ and $\mu_{i,t}(v_s,j) > i$ when $t > i$.
	   	   \end{itemize}
	   	   So $i \notin Im(\mu_{i,t})$ and hence
	   	 $$ Im(\mu_{i,t})  = ([2q+1] \cup \{t\})-\{i\}.$$ 
	   	 Moreover, $$Im(\phi) = [p],  \ Im(\theta_{i,t}) = \{i,t\}.$$
	   \end{observation}
	    
	   In the construction of $H$, the vertices and edges of $H$ are added in four steps.  Now we show that $yy' \in E(H)$ implies that $yy' \in E(K_c^G)$ according to these four steps.

	    \medskip
	    
	(1)    For any $i \ne j$, $Im(g_i) \cap Im(g_j) = \{i\} \cap \{j\} = \emptyset$. So  $\{g_i: i \in [c]\}$   induces a clique. 
	   
	   \medskip
	   
	   (2)
	   For $j > p$, $Im(\phi) \cap Im(g_j) = \emptyset$. So $\phi \sim g_j$ in $ K_c^G$.

	   \medskip
	   
	   (3)
	   For $i \in [p]$ and $t \in [q+2, 3q+2]$, for $j \in \{i\} \cup ([2q+2, 3q+2] - \{t\})$,
	   $Im(\mu_{i,t}) \cap Im(g_j) = \emptyset$. So  $\mu_{i,t} \sim   g_j$   in $ K_c^G$. 
	   
	   Now we show that if $t \ne t'$, then $\mu_{i,t} \sim  \mu_{i, t'} $ in $ K_c^G$. Assume 
	   to the contrary that $\mu_{i,t} \not\sim \mu_{i, t'} $ in $ K_c^G$. Then there exists $(v_s,j) (v_{s'},j') \in E(G)$ such that 
	   $\mu_{i,t} ( v_s,j) = \mu_{i, t'} (v_{s'}, j')$. Let $\alpha = \mu_{i,t} ( v_s,j) = \mu_{i, t'} (v_{s'}, j')$. Then
	   $$\alpha \in  \{j+\delta_{j \ge i}, q+j+ \delta_{q+j \ge i}, t-\delta_{i \ge t}\} \cap  \{j'+\delta_{j' \ge i}, q+j'+\delta_{q+j' \ge i}, t'-\delta_{i \ge t'}\}.$$
	    
	    For $a \in \{t,t'\}$ and $b \in \{j, j'\}$, since $a \ge q+2$ and $b \le q$, we have $\delta_{i \ge a} + \delta_{b \ge i} \le 1$.  Hence 
	      $ b + \delta_{b \ge i}  <  a - \delta_{i \ge a}$.  
	      
	   We consider two cases.
	   
	   \noindent{\bf Case 1} $j \ne j'$.
	   
	   In this case,  $j+ \delta_{j\ge i} \notin \{   j'+ \delta_{j' \ge i},   q+j'+\delta_{q+j' \ge i},  t' - \delta_{i \ge t'}\}$ and  $j'+ \delta_{j'\ge i} \notin \{   j+ \delta_{j \ge i},   q+j+\delta_{q+j \ge i},  t - \delta_{i \ge t}\}$.
	   So
	   $$\alpha \in   \{  q+j+ \delta_{q+j \ge i}, t-\delta_{i \ge t} \} \cap  \{  q+j'+\delta_{q+j' \ge i}, t'-\delta_{i \ge t'}\} .$$ 
As $t -\delta_{i \ge t} \ne t'-\delta_{i \ge t'}$ and $ q+j+ \delta_{q+j \ge i} \ne q+j'+\delta_{q+j' \ge i}$, we may assume that 
 $\alpha =  q+j+ \delta_{q+j \ge i} = t'-\delta_{i \ge t'}$. This implies that   
$d_F(v_i, v_s)=1$ and $d_F(v_{s'}, v_i) \ge 3$, in contrary to the assumption that $(v_s,j) (v_{s'},j') \in E(G)$.

	   \noindent{\bf Case 2} $j = j'$.
	   
	   As $(v_s,j) \ne (v_{s'}, j)$, we have  $v_s \ne v_{s'}$,  $v_sv_{s'} \in E(F)$, and $\alpha \in   \{j+\delta_{j \ge i}, q+j+ \delta_{q+j \ge i} \}$.
	   If $\alpha=j+\delta_{j \ge i}$, then 
	   $d_F(v_s, v_i), d_F(v_{s'}, v_i) \in \{0,2\}$. As $v_sv_{s'} \in E(F)$, this implies that $F$ has a $3$-cycle or a $5$-cycle, contrary to the assumption that $F$ has odd girth $7$. 
	   
	   Assume $\alpha =  q+j+ \delta_{q+j \ge i}$. 
	   It is possible that $\alpha \in \{t - \delta_{i \ge t}, t' - \delta_{i \ge t'} \}$.   
	   As $t - \delta_{i \ge t} \ne t' - \delta_{i \ge t'}$, by symmetry, we assume that $ q+j+ \delta_{q+j \ge i} \ne t-\delta_{i\ge t}$. Then 
	   $\mu_{i,t}(v_s, j)$ is defined by the second line of the definition of $\mu_{i,t}$. Hence  $d_F(v_s, v_i)=1$. As $v_sv_{s'} \in E(F)$, we conclude that $d_F(v_{s'}, v_i) \le 2$. This implies that 
	   $\mu_{i,t}(v_{s'}, j)$ is not defined by the third line of the definition of $\mu_{i,t}$, even though
	   $\mu_{i,t}(v_{s'}, j)$ possibly equals $t' - \delta_{i \ge t'}$. So $\mu_{i,t}(v_{s'}, j)$
	   is also defined by the second line of the definition of $\mu_{i,t}$. Hence $d_F(v_{s'}, v_i) = 1$. Then $F$ has a 3-cycle, again a contradiction.

	   \medskip
	   
	   (4) For $i \in [p]$ and $t \in [2q+2,3q+2]$, as $Im(\theta_{i,t}) = \{i,t\}$, we know that $\theta_{i,t} \sim g_j$ in $K_c^G$ for $j \not\in \{i,t\}$. It remains to show that $\theta_{i,t} \sim \mu_{i,t}$   and 
	   $\theta_{i,t} \sim  \phi$ in $ K_c^G$.
	   
	   If $\theta_{i,t} \not\sim \mu_{i,t} $ in $ K_c^G$, then there exists $(v_s,j)(v_{s'},j') \in E(G)$ such that $\theta_{i,t} (v_s,j) =  \mu_{i,t}(v_{s'},j')   \in Im(\theta_{i,t}) \cap Im(\mu_{i,t}) = \{t\}$.  But $\theta_{i,t} (v_s,j) = t$ implies that $d_F(v_s,v_i) \le 1$, and $\mu_{i,t}(v_{s'},j') = t$ implies that $d_F(v_{s'}, v_i) \ge 3$ (note that as $i \le p \le 2q+1$ and  $t \ge 2q+2$, $\mu_{i,t} (v_{s'}, j') = t-\delta_{i \ge t} = t$ when $d_{F}(v_{s'},v_{i}) \ge 3$), in contrary to the assumption that $(v_s,j)(v_{s'},j') \in E(G)$. 
	   
	   If $\theta_{i,t} \not\sim \phi $ in $ K_c^G$, then there exists $(v_s,j)(v_{s'},j') \in E(G)$ such that $\theta_{i,t} (v_s,j) =  \phi(v_{s'},j')   \in Im(\theta_{i,t}) \cap Im(\phi) = \{i\}$.   But $\theta_{i,t} (v_s,j)=i$ implies that $d_F(v_s, v_i  ) \ge 2$, and $\phi(v_{s'},j') = i$ implies that $s'=i$,  again in contrary to the assumption that $(v_s,j)(v_{s'},j') \in E(G)$. 
	   
	   This completes the proof of Theorem \ref{thm-main}.
	    \end{proof}

	\section{Some remarks and Questions}
	\label{sec4}

	(1) As observed in Section 3, the existence of a graph $F$ of odd girth $7$, order $p$  and fractional chromatic number greater than $3+ 4/(p-1)$ is guaranteed by a classical result of Erd\H{o}s \cite{Erdos}. Let $p$ be the minimum order of an odd girth $7$ graph $F$ with $\chi_f(F) > 3+4/(p-1)$. 
	The probabilistic proof in \cite{Erdos}  can be used to derive  an upper bound on $p$.
	  However, that bound is very big. 
		The first relatively small upper bound  on $p$, which I learned from    Anna Gujgiczer and G\'{a}bor Simonyi \cite{Simonyi},	 is obtained  by using the generalized Mycielski construction (cf. \cite{GJS2004,Tardif3}).	Let   $P_r$ be obtained from the path $v_0v_1\ldots v_r$   by adding a loop at vertex $v_0$. The {\em generalized Mycielski graph} $M_r(G)$ of $G$ is obtained from $G \times P_r$ by identifying all the vertices whose second coordinate is $v_r$. So  $M_1(G)$ is obtained from $G$ by adding a universal vertex, and $M_2(G)$ is the Mycielski graph of $G$. Tardif \cite{Tardif3} showed that for any graph $G$ and any positive integer $r$, $$\chi_f(M_r(G)) = \chi_f(G) + \frac{1}{\sum_{i=0}^{r-1}(\chi_f(G)-1)^i}.$$ 
		For an integer vector $\vec{r} = (r_1,r_2,\ldots, r_d)$, let $M_{\vec{r}}(G) = M_{r_1}(M_{r_2} \ldots (M_{r_d}(G) \ldots))$. Let $G = M_{(3,3,3,3)}(C_7)$. 
		Then $G$ has 607 vertices, odd girth $7$, and $\chi_f(G) > 3.09$ according to the above formula.  
		
		After the preliminary version of this paper, I learned from  Geoffrey Exoo and Jan Goedgebeur \cite{Exoo} that they found an odd girth 7 graph $G$ on 49 vertices with independence number 17, through a computer search. Then $M_3(G)$ is an odd girth 7 graph on $148$ vertices with fractional chromatic number greater than $3.03$.
		Later, Geoffrey Exoo found an odd girth 7 graph $G$ on 83 vertices and with independence number 27. This implies that $\chi_f(G) > 3.074$. 
		 So $p \le 83$.   
		The exact value of  $p$   remains an open problem. 
		
		\begin{question}
			\label{q1}
			What is the minimum order $p$ of a graph $F$ of odd girth $7$ and with $\chi_f(F) > 3+ 4/(p-1)$? In general, what is the minimum order $n_{k,t}$ of a graph with   odd girth at least $k$ and fractional chromatic number at least $t$?
		\end{question}

		With $c =3q+2$ and $q = \lceil \frac{p-1}{2} \rceil$, the number of vertices in $H$ is  
		$$3q+2+1 +p(2q+1)+p(q+1) = 
	  3  \lceil \frac {p+1}2 \rceil   (p+1)-p.$$ The graph $G=F[K_q]$ has $pq$ vertices.   For $p=83$, we have $q=41$ and $c=125$. The two graphs $G$ and $H$ in Theorem \ref{thm-main}  have $3,403$ and $10,501$ vertices, respectively.
	  
	  Exoo also found an odd girth 7 graph $G$ on $84$ vertices and with independence number $27$. Then $\chi_f(G) \ge 3+1/9$. In the proof of Theorem \ref{thm-main}, $q$ can be any integer at least $(p-1)/2$. Also instead of $c=3q+2$, we may choose $c=3q+3$ or $3q+4$, provided that we can guarantee that $\chi(G) \ge \chi_f(F)q > c$. By choosing $q$ and $c$ appropriately, we can show that Hedetniemi's conjecture fails for all $c \ge 125$, except that for $c=127$ (if we choose $p=83, q=41, c=3q+4$, then we cannot guarantee that $\chi(G) > c$). 

	\bigskip
	
	(2) The Poljak-R\"{o}dl function is defined in \cite{PR}:
	$$f(n) = \min\{\chi(G \times H): \chi(G), \chi(H) \ge n\}.$$
	Hedetniemi's conjecture is equivalent to saying that $f(n) =n$ for all positive integer $n$. Shitov's Theorem says that for sufficiently large $n$, $f(n) \le n-1$. Using Shitov's result,
	Tardif and Zhu \cite{TZ2019} proved that $f(n) \le n - \left(\log n\right)^{1/4-o(1)}$ for sufficiently large $n$. Tardif and Zhu asked in \cite{TZ2019} if there is a positive constant $\epsilon $ such that $f(n) \le (1-\epsilon)n$ for sufficiently large $n$. This question was answered in affirmative by He and Wigderson \cite{HW2019} with $\epsilon \approx 10^{-9}$.  Recently, Zhu    \cite{Zhu-Poljak-Rodl} proved that  $\limsup_{n \to \infty} \frac{f(n)}{n} \le \frac 12$.

		  On the other hand, the only known lower bound for $f(n)$ is that $f(n) \ge 4$ for $n \ge 4$. An intriguing question is whether   $f(n)$ goes to infinity with $n$.
		  
		  \begin{question}
		  	[\cite{PR}] Is $f(n)$ bounded by a constant or $\lim_{n \to \infty} f(n) = \infty$?
		  \end{question}
		  
		 It was proved in \cite{PR} that if $f(n)$ is bounded by a constant, then the smallest such constant is at most $16$. This result is further improved in \cite{Poljak} (see also \cite{Zhu}) that the smallest such constant is at most $9$.
		 To gain insight to this problem, it is still  interesting to study Hedetniemi's conjecture for small integers $c$.

		 The fractional version of Hedetniemi's conjecture was proved in  \cite{Zhu-frac}:
		For any two graphs $G$ and $H$, $\chi_f(G \times H) = \min \{\chi_f(G), \chi_f(H)\}$.
		Thus if $f(n)$ is bounded by $9$, and  $G$ and $H$ are $n$-chromatic graphs with 
		$\chi(G \times H) \le 9$, then at least one of $G$ and $H$  has   fractional chromatic number at most $9$.

		In \cite{Zhu-frac},  the following Poljak-R\"{o}dl type function was defined:
		$$\psi(n) = \min \{\chi(G \times H): \chi_f(G), \chi(H) \ge n\}.$$
		It follows from the definition that 
		$f(n) \le \psi(n)$.
		This author proposed a weaker version of Hedetniemi's conjecture in \cite{Zhu-frac}, which is equivalent to the statement that 
		$\psi(n)=n$ for all positive integer $n$. However, Shitov's proof actually refutes this weaker version of Hedetniemi's conjecture. The same is true for the result in this paper,   as the graph $G$ used in the proof of Theorem \ref{thm-main}  has large fractional chromatic number.

		The proof in \cite{Zhu-Poljak-Rodl} actually    shows that 
		$$\limsup_{n \to \infty} \frac{\psi(n)}{n} \le \frac 12.$$
	 A natural question is the following: 
		
		\begin{question}
			 Is $\psi(n)$ bounded by a constant?
			 If $\psi(n)$ is bounded by a constant, what could be  the smallest such constant?
		\end{question}

		\bigskip

  \bigskip

  (3) The result of El-Zahar and Sauer that Hedetniemi's conjecture holds for $c=3$ is still the best result in the positive direction of Hedetniemi's conjecture.  
  On the other hand, there are nice strengthenings of this result, in the setting of multiplicative graphs. We say a graph $Q$ is \emph{multiplicative} if for any two graphs $G,H$, $G \not\to Q$ and $H \not\to Q$ implies that $G \times H \not\to Q$.   Hedetniemi's conjecture is equiavelnt to the statement that $K_c$ is multiplicative for any positive integer $c$. El-Zahar and Sauer proved that $K_3$ is multiplicative. H\"{a}ggkvist, Hell, Miller and Neumann Lara \cite{HHMN} proved that odd cycles are multiplicative and Tardif \cite{Tardif2} proved that circular cliques $K_{k/d}$ for $k/d< 4$ are multiplicative, where $K_{k/d}$ has vertex set $[k]$ with $i\sim j$ if and only if $d \le |i-j| \le k-d$. (So $K_{k/1}=K_k$ and $K_{(2k+1)/k}=C_{2k+1}$ ).  In 2017, Wrochna \cite{Wrochna}  extended greatly the family of known multiplicative graphs by showing that any graph without $4$-cycles is multiplicative. This result is further improved in \cite{Tardif-W} where it is 
  shown that graphs in which each edge lies in at most one 4-cycle are multiplicative, and also that third powers of graphs with girth at least 13 are multiplicative.

		\vspace{3mm}
		
	\noindent
	{\bf Acknowledgement} I would like to thank G\'{a}bor Tardos for inspiring comments that led to further reduction on the chromatic number and the number of vertices of the involved graphs from an earlier version, and for many comments that improve the presentation of this paper. 
	I also  thank  Eun-Kyung Cho, Geoffrey Exoo, Jan Goedgebeur, 
	Anna Gujgiczer,
	Benjamin R. Moore,
	Yaroslav Shitov,  
	 G\'{a}bor Simonyi and 
	  Claude Tardif 
	     for valuable comments and discussions and for the efforts of finding and reducing   the upper bounds for $p$ and keep me informed.

\end{document}